\documentclass[11pt, a4paper]{amsart}
\usepackage{amssymb}
\usepackage{bbm}

\usepackage[pdftex]{hyperref}

\usepackage{url}

\newtheorem{theorem}{Theorem}[section]
\newtheorem{corollary}[theorem]{Corollary}
\newtheorem{proposition}[theorem]{Proposition}
\newtheorem{lemma}[theorem]{Lemma}

\theoremstyle{definition}
\newtheorem{definition}[theorem]{Definition}
\newtheorem{example}[theorem]{Example}

\theoremstyle{remark}

\newtheorem{notation}[theorem]{Notation}

\numberwithin{equation}{section}

\title{Binomial canonical decompositions of binomial ideals}
\author[I. Ojeda]{Ignacio Ojeda}
\address{Universidad de Extremadura, Departamento de Matem\'{a}ticas, Facultad de Ciencias, Avenida de Elvas, s/n. 06071-Badajoz (SPAIN)}
\email{ojedamc@unex.es}

\thanks{The author is partially supported by the project MTM2007-64704, National Plan I+D+I}


\begin{document}

\begin{abstract}
In this paper, we prove that every binomial ideal in a polynomial ring over an algebraically closed field of characteristic zero admits a canonical primary decomposition into binomial ideals. Moreover, we prove that this special decomposition is obtained from a cellular decomposition which is also defined in a canonical way and does not depend on the field.
\end{abstract}

\maketitle

\section*{Introduction}

It is well known that in general an ideal of a commutative Noetherian ring does not have a unique minimal primary decomposition; for example, the ideal $\langle x^2, xy \rangle \subset \mathbbmss{C}[x,y]$ has infinitely many minimal primary decompositions: $\langle x^2, xy \rangle = \langle x \rangle \cap \langle x^2,xy,y^m \rangle,\ m \geq 1.$ However, it is possible to define a primary decomposition with uniqueness property. This primary decomposition is due to V.~Ortiz (\cite{Ortiz}) and is called the canonical decomposition (see Theorem \ref{Th Ortiz}).

On the other hand, if $I$ is a binomial ideal in a polynomial ring $S$ over an algebraically closed field $\mathbbmss{k},$ there exists a primary decomposition of $I$ into binomial ideals, where by binomial ideal we mean an ideal of $S$ generated by polynomials with at most two terms.

However, the primary components in the canonical decomposition of a binomial ideal are not necessarily binomial (see Example \ref{Ejem CDnB1}). So, the initial motivation of this work was to answer the following question: is it possible to define a canonical primary decomposition of a binomial ideal in terms of binomial ideals?

Theorem \ref{Th CBPD} provides an affirmative answer when the characteristic of the field is zero. This result is interesting but not very surprising (see the comment after Corollary \ref{Cor BCD1}). The main result in Section 3 is, in fact, Theorem \ref{Th BCPD2} which states that the binomial canonical decomposition of a binomial ideal is univocally determined by an intermediate and unique decomposition introduced in Section \ref{Sect CCD} that we have called the ``canonical cellular decomposition''.

\medskip
This paper is organized as follows. In Section \ref{Sect CPD} we state without proof the theorem of existence and uniqueness of canonical decompositions (Theorem \ref{Th Ortiz}) and explore some of its consequences, especially interesting is the linear growth property of the canonical decompositions of powers of an ideal in a commutative Noetherian ring (Theorem \ref{Th Main1}). In Section \ref{Sect CCD} we proceed with the study of cellular decompositions of an ideal $I$ in a polynomial ring $S$ (see definitions \ref{Def Cellular} and \ref{Def CD}). Cellular decompositions were first introduced by D.~Eisenbud and B.~Sturmfels in \cite{Eisenbud96} as a tool for computing the associated primes and also the primary components of a binomial ideal.  The advantage of using these decompositions lies in the facts that they always exist, do not depend on the field and can be computed efficiently (e.g. by adapting \cite[Algorithm 2]{OjPie}). So, a natural question arises: is there a Ortiz-type theorem for cellular decompositions? The affirmative answer is given by Theorem \ref{Th CCD}, in fact, we prove that the canonical cellular decomposition is the canonical (primary) one if, and only if, every cellular canonical component is primary (Theorem \ref{Prop CCD=CD}). Finally in this section, we prove that, if the characteristic of the field is zero, the canonical cellular components of a binomial ideal are binomial (Theorem \ref{Th BCCD}). In Section \ref{Sect BCD}, the main results on the binomial canonical decomposition mentioned above are stated and proved. Finally, in Section \ref{Sect Examples}, some relevant examples of canonical decompositions are shown.

\medskip
It is worth to pointing out that the study of primary decomposition of binomial ideals has recently attracted the attention of many researchers (see, e.g. \cite{DMM, Kahle}), motivated in part by the use of primary decomposition in the context of the so-called Algebraic Statistics. We hope that this work may stimulate the use of the primary decomposition in this and other research areas.

\section{Canonical primary decomposition}\label{Sect CPD}

Throughout this section $R$ will denote a commutative Noetherian ring.

We begin by recalling the notion index of nilpotency of an ideal of $R$ which will be extensively used in this paper.

\begin{definition}
The index of nilpotency of $I, \text{nil}(I),$ is the smallest integer $e$ such that $$\left( \sqrt{I} \right)^e \subseteq I.$$
\end{definition}

Some authors call $\text{nil}(I)$ the degree of nilpotency or the exponent of $I$ (see, e.g. \cite[Section 9.2]{Vasconcelos}).

The next result due to V. Ortiz \cite{Ortiz} establishes the existence of a canonical primary decomposition of ideals in a commutative Noetherian ring.

\begin{theorem}\label{Th Ortiz}
Every ideal $I$ in $R$ admits a unique minimal primary decomposition: $$I = Q_1^* \cap Q_2^* \cap \ldots \cap Q_t^*,$$ such that if $I = Q_1 \cap Q_2 \cap \ldots \cap Q_t$ is another minimal primary decomposition of $I,$ then
\begin{itemize}
\item[(a)] ${\rm nil}(Q_i^*) \leq {\rm nil}(Q_i),\ i=1, \ldots, t;$
\item[(b)] if ${\rm nil}(Q_i^*) = {\rm nil}(Q_i),$ then $Q_i^* \subseteq Q_i.$
\end{itemize}
\end{theorem}

\begin{proof}
For a proof see \cite{Ortiz} or \cite[Theorem 6.2]{Swanson07}.
\end{proof}

The primary ideals $Q^*_i,\ i = 1, \ldots t,$ are called canonical components of $I$ and $Q_1^* \cap Q_2^* \cap \ldots \cap Q_t^*$ is called the canonical decomposition of $I.$

\medskip
As immediate consequences we have the following:

\begin{corollary}\label{Cor Pre0}
If $Q^*$ is the $P-$canonical component of an ideal $I$ of $R,$ then $Q^*$ is equal to the $P-$primary component of $I+P^{\mathrm{nil}(Q^*)}.$
\end{corollary}

The above result was already noticed by V.~Ortiz in \cite{Ortiz}.

\begin{corollary}\label{Cor Pre1}
Let $\cap_{i=1}^t Q_i$ be any minimal primary decomposition of an ideal $I$ of $R.$ The index of nilpotency of the $\sqrt{Q_j}-$canonical component of $I$ is the smallest integer $e_j$ such that $$I = \Big(I + \big(\sqrt{Q_j}\big)^{e_j}\Big) \cap \big(\bigcap_{i \neq j} Q_i\big).$$
\end{corollary}

\begin{proof}
It suffices to note that the $\sqrt{Q_j}-$primary component of $I + \big(\sqrt{Q_j}\big)^{e_j}$ is a $\sqrt{Q_j}-$primary component of $I$ whose index of nilpotency is less than or equal to $e_j.$
\end{proof}

Several upper bounds for the index of nilpotency of ideals in a polynomial ring are known (see e.g. the introduction of \cite{OjPie2}). Thus, the above corollary may be considered as a naive algorithm to compute the canonical decomposition of an ideal in a polynomial ring (see \cite[Algorithm 2.6]{OjPie3}).

Let us see how this algorithm works on an example.

\begin{example}\label{Ejem CDnB1}
Let $I = \langle z^2(x-y), z^3 \rangle \subset \mathbbmss{C}[x,y,z].$ Cleary, $I = \langle z^2 \rangle \cap \langle x-y, z^3 \rangle$ is minimal primary decomposition of $I.$ In this case, since $\mathrm{nil}(\langle z^2 \rangle) = 2$ and $\mathrm{nil}(\langle x-y, z^3 \rangle) = 3,$ we have that the indices of nilpotency of the corresponding canonical components are less than or equal to $2$ and $3,$ respectively. Of course, we already know that $\langle z^2 \rangle$ is the $\langle z \rangle-$canonical component of $I$ (because, $\langle z \rangle$ is a minimal prime of $I$); on the other hand, since $\langle z^2 \rangle \subseteq I + \langle x-y, z \rangle^2,$ by Corollary \ref{Cor Pre1}, we have that the index of nilpotency of the $\langle x-y, z \rangle-$canonical component of $I$ is $3.$ Thus, by Corollary \ref{Cor Pre0}, we conclude that the other canonical component is the minimal primary component of $Q = I + \langle x-y, z \rangle^3$ which, in this case, coincides with $Q$ itself.

Observe that $Q$ is not a binomial ideal, this can be checked by direct computation using \cite[Proposition 1.1]{Eisenbud96}.
\end{example}

\medskip
To show the potential of the canonical decomposition, we finish this preliminary section by using it to rephrase the following result on the linear growth of primary decompositions of power of an ideal.

\medskip
\noindent\textbf{Theorem.} (I.~Swanson, \cite{Swanson97}).
\emph{Let $I$ be an ideal of $R.$ There exists an integer $k$ such that for all $n \geq 1$ there exists a primary decomposition $I^n = Q_1 \cap \ldots \cap Q_t$ such that $$\big( \sqrt{Q_j} \big)^{kn} \subseteq Q_j,$$ for all $j = 1, \ldots, t.$}

\begin{theorem}\label{Th Main1}
Let $R$ be a commutative Noetherian ring and let $I$ be an ideal of $R.$ There exists an integer $k$ such that for all $n \geq 1$ $$\mathrm{nil}(Q^*) \leq kn$$ for every canonical component $Q^*$ of $I^n.$
\end{theorem}

The proof of Theorem \ref{Th Main1} follows immediately from the results introduced in \cite{Yao}, for the same purpose.

\begin{proof}
By \cite[Theorem 3.3]{Yao}, there exists $k \in \mathbb{N}$ such that $$I^n = (I^n+J^{kn}) \cap (I^n : J^\infty),$$ for all $n \in \mathbb{N}$ and for all ideals $J \subseteq R.$ So, if $P$ is an associated prime ideal of $I^n,$ we have that $I^n = (I^n+P^{kn}) \cap (I^n : P^\infty).$ Therefore, since the $P-$primary component of $I^n+P^{kn}$ is a $P-$primary component of $I^n$ with index of nilpotency is less than or equal to $kn,$ we conclude that the index of nilpotency of the $P-$canonical component of $I^n$ is less than or equal to $kn.$
\end{proof}

\section{Canonical cellular decomposition}\label{Sect CCD}

Let $\mathbbmss{k}[\mathbf{t}] = \mathbbmss{k}[t_1, \ldots, t_n]$ be the polynomial ring in $n$ variables over an arbitrary field $\mathbbmss{k}.$

In what follows, given $\delta \subseteq \{1, \ldots, n\},$ we will denote by $\mathfrak{m}_\delta$ the monomial prime ideal $\langle t_j \mid j \not\in \delta \rangle \subseteq \mathbbmss{k}[\mathbf{t}]$ (by convention, if $\delta = \{1, \ldots, n\},$ then $\mathfrak{m}_\delta = \langle 0 \rangle$) and we will write $\mathbf{t}_\delta$ for $\prod_{j \in \delta} t_j.$

\begin{definition}\label{Def Cellular}
We define an ideal $I$ of $\mathbbmss{k}[\mathbf{t}]$ to be cellular if either $I = \langle 1 \rangle$ or $I \neq \langle 1 \rangle$ and, for some $\delta \subseteq \{1, \ldots, n\},$ we have that
\begin{itemize}
\item[1.] $I = \big( I : \mathbf{t}_\delta^\infty \big),$
\item[2.] there exists a positive integer $e$ such that $\mathfrak{m}^e_\delta \subseteq I;$
\end{itemize}
in this case, we say that $I$ is cellular with respect to $\delta$ or, simply, $\delta-$cellular.
\end{definition}

Observe that an ideal $I$ of $\mathbbmss{k}[\mathbf{t}]$ is cellular if, and only if, every variable of $\mathbbmss{k}[\mathbf{t}]$ is either a nonzerodivisor or nilpotent modulo $I.$ In particular, every primary ideal is cellular.

\begin{definition}\label{Def CD}
A cellular decomposition of an ideal $I \subseteq \mathbbmss{k}[\mathbf{t}]$ is an expression of $I$ as an intersection of cellular ideals with respect to different $\delta \subseteq \{1, \ldots, n\},$ say
\begin{equation}\label{ecu cell0} I = \bigcap_{\delta \in \Delta} \mathcal{C}_\delta,\end{equation} for some subset $\Delta$ of the power set of $\{1, \ldots, n\}.$ If moreover we have $\mathcal{C}_\delta' \not\supseteq \bigcap_{\delta \in \Delta \setminus \{\delta'\}} \mathcal{C}_\delta$ for every $\delta' \in \Delta,$ the cellular decomposition (\ref{ecu cell0}) is said to be minimal; in this case, the cellular component $C_\delta$ is said to be a $\delta-$cellular component of $I.$
\end{definition}

\begin{example}
Every minimal primary decomposition of a monomial ideal $I \subseteq \mathbbmss{k}[\mathbf{t}]$ into monomial ideals is a minimal cellular decomposition of $I.$
\end{example}

Cellular decompositions of an ideal $I$ of $\mathbbmss{k}[\mathbf{t}]$ always exist. A simple algorithm for cellular decomposition of binomial ideals can be found in \cite[Algorithm 2]{OjPie};
however, since this algorithm does not actually require a binomial input, it can be also used to compute a cellular decomposition of a (not necessarily binomial) ideal of $\mathbbmss{k}[\mathbf{t}].$ The interested reader may consult \cite{OjPie} or \cite{Kahle} for the details.

Algorithm 2 in \cite{OjPie} forms part of the \texttt{Binomials} package developed by T.~Kahle and is publicly available at
\begin{center}
\url{http://personal-homepages.mis.mpg.de/kahle/bpd/}
\end{center}

\medskip
Now, we will show that every ideal of $\mathbbmss{k}[\mathbf{t}]$ has a canonical cellular decomposition.

\begin{lemma}\label{Lemma Cell0}
Let $I$ be an ideal of $\mathbbmss{k}[\mathbf{t}]$ and let $I = \bigcap_{\delta \in \Delta} \mathcal{C}_\delta$ be a cellular decomposition of $I.$ If $\delta_0 \in \Delta$ is minimal with respect to inclusion, then $$\big( I : \mathfrak{m}_{\delta_0}^\infty \big) = \bigcap_{\delta \in \Delta \setminus \{\delta_0\}} \mathcal{C}_\delta.$$ In particular, the ideal $\bigcap_{\delta \in \Delta \setminus \{\delta_0\}} \mathcal{C}_\delta$ is independent of the particular decomposition of $I.$
\end{lemma}

\begin{proof}
Due to the minimality of $\delta_0,$ for each $\delta \in \Delta \setminus \{\delta_0\},$ there is, at least, a variable in $\{t_i \mid i \not\in \delta_0\}$ which is a nonzerodivisor modulo $\mathcal{C}_\delta.$  Therefore, $$\big( C_\delta : \mathfrak{m}_{\delta_0}^\infty \big) = \left\{\begin{array}{ccc} C_\delta & \text{if} & \delta \neq \delta_0\\ \langle 1 \rangle & \text{if} & \delta = \delta_0\end{array}\right.$$ and our claim follows.
\end{proof}

\begin{theorem}\label{Th Cell0}
Let $I$ be an ideal of $\mathbbmss{k}[\mathbf{t}]$ and let $I = \bigcap_{\delta \in \Delta} \mathcal{C}_\delta$ be a minimal cellular decomposition of $I.$ Then the subset $\Delta$ of the power set of $\{1, \ldots, n\}$ is independent of the particular decomposition of $I.$
\end{theorem}

\begin{proof}
We proceed by induction on the cardinality on $\Delta.$ Of course, if $\# \Delta = 1, I$ is cellular and there is nothing to prove. Otherwise, we consider any other minimal cellular decomposition of $I,$ say $I = \bigcap_{\delta' \in \Delta'} \mathcal{C}'_{\delta'}.$ Let $\delta_0 \in \Delta \cup \Delta'$ be minimal with respect to inclusion, without loss of generality, we may assume $\delta_0 \in \Delta.$ By Lemma \ref{Lemma Cell0}, we have $$ \bigcap_{\delta \in \Delta \setminus \{\delta_0\}} \mathcal{C}_\delta = (I : \mathfrak{m}_{\delta_0}^\infty) = \bigcap_{\delta' \in \Delta' \setminus \{\delta_0\}} \mathcal{C}'_\delta.$$ If $\delta_0 \not\in \Delta',$ the right-most term in the above equalities is equal to $I;$ so, $I = \bigcap_{\delta \in \Delta \setminus \{\delta_0\}} \mathcal{C}_\delta$ and therefore $ \bigcap_{\delta \in \Delta \setminus \{\delta_0\}} \mathcal{C}_\delta \subseteq \mathcal{C}_{\delta_0}$ in clear contradiction with the minimality of the cellular decomposition $I = \bigcap_{\delta \in \Delta} \mathcal{C}_\delta.$ Thus, we have that $\delta_0 \in \Delta'.$ Now, since $(I : \mathfrak{m}_{\delta_0}^\infty)$ does not depend on the chosen cellular decompositions, we conclude by induction hypothesis.
\end{proof}

\begin{notation}
Let $I$ be an ideal of $\mathbbmss{k}[\mathbf{t}].$ In what follows, we will denote by $\Delta(I)$ the subset of the power set of $\{1, \ldots, n\}$ appearing in any minimal cellular decomposition of $I$ to emphasize that $\Delta(I)$ depends only on $I.$
\end{notation}

\begin{corollary}\label{Cor Cell1}(Compatibility).
Let $I$ be an ideal of $\mathbbmss{k}[\mathbf{t}]$ and set $\Delta = \Delta(I).$ If $I = \bigcap_{\delta \in \Delta} \mathcal{C}_\delta$ and $I = \bigcap_{\delta \in \Delta} \mathcal{C}'_\delta$ are two minimal cellular decompositions of $I,$ then
$$
I = \Big( \bigcap_{\delta \in \Delta_1} \mathcal{C}_\delta \Big) \cap \Big( \bigcap_{\delta \in \Delta_2} \mathcal{C}'_\delta \Big)
$$
is a minimal cellular decomposition of $I,$ for every partition of $\Delta$ into disjoint subsets $\Delta_1$ and $\Delta_2.$
\end{corollary}

\begin{proof}
We proceed by induction on the cardinality on $\Delta.$ Again if $\# \Delta = 1, I$ is cellular and there is nothing to prove. Otherwise, let $\delta_0 \in \Delta$ be minimal with respect to inclusion and define $\bar\Delta = \Delta \setminus \{\delta_0\}.$ By Lemma \ref{Lemma Cell0}, we have $$\bigcap_{\delta \in \bar\Delta} \mathcal{C}_\delta = (I : \mathfrak{m}_{\delta_0}^\infty) = \bigcap_{\delta \in \bar\Delta} \mathcal{C}'_\delta.$$ So, both $$I = \Big( \bigcap_{\delta \in \bar\Delta} \mathcal{C}_\delta \Big) \cap \mathcal{C}'_\delta\quad \text{and}\quad I = \Big( \bigcap_{\delta \in \bar\Delta} \mathcal{C}'_\delta \Big) \cap \mathcal{C}_\delta$$ are minimal cellular decompositions of $I.$ The result follows now by applying the induction hypothesis to $(I : \mathfrak{m}_{\delta_0}^\infty).$
\end{proof}

After Corollary \ref{Cor Cell1}, the proof of the following theorem is just an adaptation of the proof of Theorem \ref{Th Ortiz} given by I.~Swanson in \cite{Swanson07}, but we include it here for completeness:

\begin{theorem}\label{Th CCD}
Every ideal $I$ in $\mathbbmss{k}[\mathbf{t}]$ admits a unique minimal cellular decomposition
$$I = \bigcap_{\delta \in \Delta(I)} \mathcal{C}_\delta^*$$
such that if $I = \bigcap_{\delta \in \Delta(I)} \mathcal{C}_\delta$ is another minimal cellular decomposition, then we have
\begin{itemize}
\item[(a)] $\mathrm{nil}(\mathcal{C}_\delta^{*}) \leq \mathrm{nil}(\mathcal{C}_\delta),$ for every ${\delta \in \Delta(I)}.$
\item[(b)] If $\mathrm{nil}(\mathcal{C}_\delta^*) = \mathrm{nil}(\mathcal{C}_\delta)$ for some $\delta,$ then $\mathcal{C}_\delta^* \subseteq \mathcal{C}_\delta.$
\end{itemize}
\end{theorem}

The cellular ideals $\mathcal{C}_\delta^*$ will be called the $\delta-$cellular canonical components of $I$ and we will refer to $\bigcap_{\delta \in \Delta} \mathcal{C}_\delta^*$ as the canonical cellular decomposition of $I.$

\begin{proof}
By Corollary \ref{Cor Cell1}, it suffices to prove that for each $\delta \in \Delta(I),$ there exists  a cellular ideal $\mathcal{C}^*$ with respect to $\delta$ such that $\mathcal{C}^*$ appears as $\delta-$cellular component of $I$ some minimal cellular decomposition of $I, \mathrm{nil}(\mathcal{C}^*)$ is smallest possible, and if $\mathrm{nil}(\mathcal{C}^*) = \mathrm{nil}(\mathcal{C})$ for some $\delta-$cellular component $\mathrm{nil} (\mathcal{C})$ of $I,$ then $\mathrm{nil}(\mathcal{C}^*) \subseteq \mathrm{nil}(\mathcal{C}).$ Let $S$ be the set of all $\delta-$cellular components of $I$ with smallest possible index of nilpotency, say $e.$ Then $S$ is closed under intersections: $$\Big( \big( \bigcap_{\mathcal{C} \in S} \mathcal{C} \big) : \mathbf{t}_\delta^\infty \Big) = \bigcap_{\mathcal{C} \in S} \big( \mathcal{C} : \mathbf{t}_\delta^\infty \big) = \bigcap_{\mathcal{C} \in S} \mathcal{C}$$ and $$\mathfrak{m}_\delta^e \subseteq \bigcap_{\mathcal{C} \in S} \mathcal{C},$$ then $\bigcap_{\mathcal{C} \in S} \mathcal{C}$ is cellular with respect to $\delta$ and $\mathrm{nil}\big(\bigcap_{\mathcal{C} \in S} \mathcal{C}\big) = e.$ Thus $S$ has a minimal element under inclusion. This element, $\mathcal{C}^*,$ satisfies the two conditions of the theorem.
\end{proof}

We next derive a necessary and sufficient condition for the canonical cellular decomposition to be the canonical (primary) decomposition.

\begin{proposition}\label{Prop CCD=CD}
Let $I$ be an ideal in $\mathbbmss{k}[\mathbf{t}]$ and let $\mathcal{C}^*_\delta$ be the $\delta-$cellular canonical component of $I.$ Then $\mathcal{C}^*_\delta$ is primary if, and only if,  $\mathcal{C}^*_\delta$ is the $\sqrt{\mathcal{C}^*_\delta}-$canonical component of $I.$
\end{proposition}

\begin{proof}
If $\mathcal{C}^*_\delta$ is a primary ideal, clearly $\sqrt{\mathcal{C}^*_\delta}$ is prime. Furthermore,
$\sqrt{\mathcal{C}^*_\delta}$ is an associated prime of $I;$ otherwise, $\bigcap_{\delta' \in \Delta \setminus \delta} \mathcal{C}^*_{\delta'} \subseteq \mathcal{C}^*_\delta.$ So, there is a canonical component of $I$ whose radical is $\sqrt{\mathcal{C}^*_\delta}.$ Now, since every $\sqrt{\mathcal{C}^*_\delta}-$primary ideal is cellular with respect to $\delta,$ we conclude that $\mathcal{C}^*_\delta$ is the $\sqrt{\mathcal{C}^*_\delta}-$canonical component of $I.$

The converse is obviously true, because the canonical components of $I$ are primary.
\end{proof}

\begin{corollary}
The canonical cellular decomposition agrees with the cano\-nical (primary) decomposition if, and only if, every canonical cellular component is primary.
\end{corollary}

Some examples of ideals whose canonical cellular decomposition is the canonical (primary) one are shown in Section \ref{Sect Examples}.

\medskip
Finally, let us see that if $\mathrm{char}(\mathbbmss{k}) = 0,$ the cellular ideals appearing the canonical cellular decomposition of a binomial ideal of $\mathbbmss{k}[\mathbf{t}]$ are binomial.

\begin{theorem}\label{Th BCCD}
Let $\mathrm{char}(\mathbbmss{k}) = 0.$ If $\mathcal{C}^*_\delta$ is the $\delta-$cellular canonical component of a binomial ideal $I \subset \mathbbmss{k}[\mathbf{t}],$ then $$\mathcal{C}^*_\delta = \Big( \big( I + \mathfrak{m}_\delta^{\mathrm{nil}(\mathcal{C}^*_\delta)} \big) : \mathbf{t}_\delta^\infty \Big).$$ In particular, $\mathcal{C}^*_\delta$ is binomial.
\end{theorem}

\begin{proof}
Let $e = \mathrm{nil}(\mathcal{C}^*_\delta)$ and define $\mathcal{C} = \big( ( I + \mathfrak{m}_\delta^e ) : \mathbf{t}_\delta^\infty \big).$ First, we observe that $\mathcal{C} \subseteq \mathcal{C}^*_\delta \neq \langle 1 \rangle.$ Moreover, by construction, $\mathcal{C}$ is cellular with respect to $\delta$ and, by \cite[Corollary 1.7(b)]{Eisenbud96}, is binomial.

By \cite[Theorem 3.1]{OjPie2}, $\mathrm{nil}(\mathcal{C}) \geq e;$ furthermore, if the characteristic of $\mathbbmss{k}$ is zero, the equality holds (\cite[Corollary 3.1]{OjPie2}). Thus, in our case, $\mathcal{C}$ has the smallest index of nilpotency possible. Finally, if $\mathcal{C}'$ is another $\delta-$cellular component of $I$ with $\mathrm{nil}(\mathcal{C}') = e,$ then $I \subseteq \mathcal{C}'$ and $\mathfrak{m}_\delta^e \subseteq \mathcal{C}',$ and so $$\mathcal{C} = \big( ( I + \mathfrak{m}_\delta^e ) : \mathbf{t}_\delta^\infty \big) \subseteq (\mathcal{C}' : \mathbf{t}_\delta^\infty ) = \mathcal{C}'.$$ Therefore, by Theorem \ref{Th CCD}, we have that $\mathcal{C}$ is the $\delta-$cellular canonical component of $I.$
\end{proof}

\begin{corollary}\label{Cor BCCD1}
Let $\mathrm{char}(\mathbbmss{k}) = 0.$ If $I = \bigcap_{\delta \in \Delta} \mathcal{C}_\delta$ is a minimal cellular decomposition of a binomial $I \subset \mathbbmss{k}[\mathbf{t}],$ the index of nilpotency of the $\delta'-$cellular canonical component of $I$ is the smallest integer $e_{\delta'}$ such that
$$
I = \Big( \big( I + \mathfrak{m}_{\delta'}^{e_{\delta'}} \big) : \mathbf{t}_{\delta'}^\infty \Big) \cap \Big(\bigcap_{\delta \in \Delta \setminus \delta'} \mathcal{C}_\delta \Big).
$$
\end{corollary}

\begin{proof}
It is an immediate consequence of Corollary \ref{Cor Cell1} and Theorem \ref{Th BCCD}.
\end{proof}

\section{Binomial canonical decomposition}\label{Sect BCD}

From now on, we will assume that $\mathbbmss{k}$ is an algebraically closed field of characteristic zero.

\begin{theorem}\label{Th CBPD}
Every binomial ideal $I \subset \mathbbmss{k}[\mathbf{t}]$ admits a unique minimal primary decomposition into binomial ideals: $$I = \bigcap_{i=1}^t Q^{(*)}_i$$ such that if $I = \bigcap_{i=1}^t Q_i$ is another minimal primary decomposition of $I$ into binomial ideals, then
\begin{itemize}
\item[(a)] ${\rm nil}(Q^{(*)}_i) \leq {\rm nil}(Q_i),\ i=1, \ldots, t;$
\item[(b)] if ${\rm nil}(Q^{(*)}_i) = {\rm nil}(Q_i),$ then $Q_i^{(*)} \subseteq Q_i.$
\end{itemize}
\end{theorem}

\begin{notation}
Here and subsequently, let $\mathbbmss{k}[\mathbf{t}_\delta]$ denote the ring $\mathbbmss{k}[t_i \mid i \in \delta],\ \delta \subseteq \{1, \ldots, n\}$ and let $(-)_P$ denote the contraction to $\mathbbmss{k}[\mathbf{t}]$ of the localization at a prime ideal $P \subseteq \mathbbmss{k}[\mathbf{t}].$ Moreover, for simplicity of notation and when no confusion is possible, we will write $(-) \cap \mathbbmss{k}[\mathbf{t}_\delta]$ for the ideal in $\mathbbmss{k}[\mathbf{t}]$ generated by  $(-) \cap \mathbbmss{k}[\mathbf{t}_\delta].$
\end{notation}

In the next lemma we collect, for future reference, some properties of the associated primes and binomial primary components of cellular binomial.

\begin{lemma}\label{Lemma BCD1}
Let $I \subset \mathbbmss{k}[\mathbf{t}]$ be a binomial ideal. Then the following holds:
\begin{itemize}
\item[(a)] If $P$ is an associated prime of $I,$ then $P = P \cap \mathbbmss{k}[\mathbf{t}_\delta] + \mathfrak{m}_\delta,$ for some $\delta \in \{1, \ldots, n\}.$
\item[(b)] If $Q$ is a binomial $P-$primary component of $I$ with $P = P \cap \mathbbmss{k}[\mathbf{t}_\delta] + \mathfrak{m}_\delta,$ then $P \cap \mathbbmss{k}[\mathbf{t}_\delta] \subseteq Q.$
\end{itemize}
\end{lemma}

\begin{proof}
For a proof of (a) and (b) see \cite[Corollary 2.6]{Eisenbud96} and the proof of \cite[Theorem 7.1'(b)]{Eisenbud96}, respectively.
\end{proof}

\medskip
\noindent\emph{Proof of Theorem \ref{Th CBPD}.}
First, we recall that, since $\mathbbmss{k}$ is algebraically closed, $I$ has a minimal primary decomposition in terms of binomial ideals by \cite[Theorem 7.1]{Eisenbud96}.

Let $Q_P$ be a binomial $P-$primary component of $I$ with the smallest possible index of nilpotency. Set $e = \mathrm{nil}(Q_P)$ and define \begin{equation}\label{ecu PD1} Q_P^{(*)} = (I + P \cap \mathbbmss{k}[\mathbf{t}_\delta] + \mathfrak{m}^e_\delta)_P.\end{equation}

On the one hand, $$P \cap \mathbbmss{k}[\mathbf{t}_\delta] + \mathfrak{m}_\delta = \sqrt{I + P \cap \mathbbmss{k}[\mathbf{t}_\delta] + \mathfrak{m}^e_\delta} \subseteq \sqrt{Q_P^{(*)}} \subseteq P.$$ Thus $\sqrt{Q_P^{(*)}} = P$ by Lemma \ref{Lemma BCD1}(a) and, consequently, $P$ is the only minimal prime of the binomial ideal $I + P \cap \mathbbmss{k}[\mathbf{t}_\delta] + \mathfrak{m}^e_\delta.$ Therefore $Q_P^{(*)}$ is $P-$primary; moreover, by \cite[Corollary 6.5]{Eisenbud96}, we have that $Q_P^{(*)}$ is a binomial ideal.

On the other hand, by Lemma \ref{Lemma BCD1}(b), $I + P \cap \mathbbmss{k}[\mathbf{t}_\delta] + \mathfrak{m}^e_\delta \subseteq Q_P.$ Thus, $I \subseteq Q_P^{(*)} \subseteq Q_P$ and we obtain that $Q_P^{(*)}$ is the binomial $P-$primary component of $I.$

Finally, since $P^e = (P \cap \mathbbmss{k}[\mathbf{t}_\delta] + \mathfrak{m}_\delta)^e \subseteq P \cap \mathbbmss{k}[\mathbf{t}_\delta] + \mathfrak{m}_\delta^e \subseteq  Q_P^{(*)},$ by the minimality of $e,$ we have that $\mathrm{nil}(Q_P^{(*)}) = e$ and we conclude that $Q_P^{(*)}$ is the binomial $P-$primary component of $I$ satisfying (a) and (b). \hfill$\qed$

\medskip
The primary ideal $Q_P^{(*)}$ described in (\ref{ecu PD1}) will be called the binomial canonical $P-$primary component of $I$ we will refer to $I = \bigcap_{P \in \mathrm{Ass}(\mathbbmss{k}[\mathbf{t}]/I)} Q^{(*)}_P$ as the binomial canonical decomposition of $I.$

\begin{corollary}\label{Cor BCD1}
If $Q_P^{(*)}$ is the binomial canonical $P-$primary component of a binomial ideal $I \subset \mathbbmss{k}[\mathbf{t}],$ then $$Q_P^{(*)} = (I + P \cap \mathbbmss{k}[\mathbf{t}_\delta] + \mathfrak{m}^{\mathrm{nil}(Q_P^{(*)})}_\delta)_P.$$
\end{corollary}

\begin{proof}
This was already proved in the proof of Theorem \ref{Th CBPD}.
\end{proof}

Observe that we have shown that the binomial canonical components are of the form of those appearing in \cite[Theorem 7.1'(b)]{Eisenbud96}, but with the smallest possible $e$ for each associated prime.

\medskip
Th rest of the section is devoted to exploring the very close relationship between the canonical cellular and the binomial canonical decompositions. This relationship can be summarized in the following form:

\begin{theorem}\label{Th BCPD2}
Let $I \subset \mathbbmss{k}[\mathbf{t}]$ be a binomial ideal. The binomial canonical decomposition of $I$ is
(after removing redundant components) the intersection of the binomial canonical decompositions of its cellular canonical components.
\end{theorem}

The key of the proof is in the following interesting lemma.

\begin{lemma}\label{Lemma BPD}
Let $I \subset \mathbbmss{k}[\mathbf{t}]$ be a (not necessarily binomial) cellular ideal with respect to $\delta.$
If $I = \bigcap_{P \in \mathrm{Ass}(\mathbbmss{k}[\mathbf{t}]/I)} Q_P$ is a minimal primary decomposition of $I$ into binomial ideals, then
$$\mathrm{nil}(I) = \max\big\{ \mathrm{nil}( Q_P ) \mid P \in \mathrm{Ass}(\mathbbmss{k}[\mathbf{t}]/I) \big\}.$$
\end{lemma}

\begin{proof}
Let $I=J_1 \cap \ldots \cap J_r$ be some decomposition, not necessarily primary or irreducible. The homomorphism of rings $ 0 \to S/I \to \prod_{i=1}^r S/J_i $ shows that $ \sqrt{I}/I \hookrightarrow \prod_{i=1}^r \sqrt{J_i}/J_i$ and therefore that ${\rm nil}(I) \leq \max_i\{ {\rm nil}(J_i)\}.$ In particular, we have $$\mathrm{nil}(I) \leq \max\big\{ \mathrm{nil}( Q_P) \mid P \in \mathrm{Ass}(\mathbbmss{k}[\mathbf{t}]/I) \big\}.$$ Conversely, set $\mathrm{nil}(I) = e.$ Since $Q_P$ is in particular a cellular binomial ideal with respect to $\delta,$ for every $P \in \mathrm{Ass}(\mathbbmss{k}[\mathbf{t}]/I),$ by Lemma \ref{Lemma BCD1}, $$\big( \sqrt{Q_P } \big)^e = (P \cap \mathbbmss{k}[\mathbf{t}_\delta] + \mathfrak{m}_\delta)^e \subseteq P \cap \mathbbmss{k}[\mathbf{t}_\delta] + \mathfrak{m}_\delta^e \subseteq I + P \cap \mathbbmss{k}[\mathbf{t}_\delta] \subseteq Q_P,$$ for every $P \in \mathrm{Ass}(\mathbbmss{k}[\mathbf{t}]/I),$ that is to say, $\mathrm{nil}( Q_P ) \leq e,$  for every $P \in \mathrm{Ass}(\mathbbmss{k}[\mathbf{t}]/I),$ and we are done.
\end{proof}

\medskip
\noindent\emph{Proof of Theorem \ref{Th BCPD2}.}
Let $I = \bigcap_{\delta \in \Delta(I)} \mathcal{C}^*_\delta$ be the canonical cellular decomposition of $I$
and, for each $\delta \in \Delta(I),$ let $\mathcal{C}^*_\delta = \bigcap_{P \in \mathrm{Ass}(\mathbbmss{k}[\mathbf{t}]/\mathcal{C}^*_\delta)} Q^{(*)}_{P, \delta}$ be the binomial canonical decomposition of $\mathcal{C}^*_\delta.$ Clearly \begin{equation}\label{ecu CBPD} I = \bigcap_{\begin{subarray}{c} P \in \mathrm{Ass}(\mathbbmss{k}[\mathbf{t}]/\mathcal{C}^*_\delta) \\ \delta \in \Delta(I) \end{subarray} } Q^{(*)}_{P, \delta} \end{equation} is a (possible nonminimal) primary decomposition of $I$ into binomial ideals. Then, taking into account that $\mathrm{Ass}(\mathbbmss{k}[\mathbf{t}]/I) \subseteq \bigcup_{\delta \in \Delta(I)}  \mathrm{Ass}(\mathbbmss{k}[\mathbf{t}]/\mathcal{C}^*_\delta)$ and $\mathrm{Ass}(\mathbbmss{k}[\mathbf{t}]/\mathcal{C}^*_\delta) \cap \mathrm{Ass}(\mathbbmss{k}[\mathbf{t}]/\mathcal{C}^*_{\delta'}) = \varnothing$ when $\delta \neq \delta',$ a minimal primary decomposition of $I$ into binomial ideals, say \begin{equation}\label{ecu CBPD2}I = \bigcap_{P \in \mathrm{Ass}(\mathbbmss{k}[\mathbf{t}]/I)} Q_P,\end{equation} is obtained after removing redundant components in (\ref{ecu CBPD}).

Let us prove that (\ref{ecu CBPD2}) is the binomial canonical decomposition of $I.$

Let $I = \bigcap_{P \in \mathrm{Ass}(\mathbbmss{k}[\mathbf{t}]/I)} Q^{(*)}_P$ be the binomial canonical decomposition of $I.$ If $\mathcal{C}_\delta$ is the intersection of all those $Q^{(*)}_P$'s which are cellular with respect to the same $\delta \in \Delta(I),$ then $$I = \bigcap_{\delta \in \Delta(I)} \mathcal{C}_\delta$$ is a minimal cellular decomposition of $I.$ If $\mathrm{nil}(\mathcal{C}_\delta) > \mathrm{nil}(\mathcal{C}^*_\delta),$ by Lemma \ref{Lemma BPD}, $\mathrm{nil}(Q^{(*)}_P) > \mathrm{nil}(Q_P)$ for some $P \in \mathrm{Ass}(\mathbbmss{k}[\mathbf{t}]/I),$ which contradicts the minimality of $\mathrm{nil}(Q^{(*)}_P).$ Then, $\mathrm{nil}(\mathcal{C}_\delta) = \mathrm{nil}(\mathcal{C}^*_\delta),$ and so $\mathcal{C}^*_\delta \subseteq \mathcal{C}_\delta,$ for every $\delta \in \Delta(I).$

Now, if $e_P = \mathrm{nil}(Q^{(*)}_P),$ then \begin{align*} I +  P \cap \mathbbmss{k}[\mathbf{t}_\delta] + \mathfrak{m}^{e_P}_\delta & \subseteq \mathcal{C}^*_\delta + P \cap \mathbbmss{k}[\mathbf{t}_\delta] + \mathfrak{m}^{e_P}_\delta \\ & \subseteq \mathcal{C}_\delta + P \cap \mathbbmss{k}[\mathbf{t}_\delta] + \mathfrak{m}^{e_P}_\delta \subseteq Q^{(*)}_P.\end{align*} Therefore, by applying the operation $(-)_P,$ we have that $$Q^{(*)}_P  = (I +  P \cap \mathbbmss{k}[\mathbf{t}_\delta] + \mathfrak{m}^{e_P}_\delta)_P  \subseteq (\mathcal{C}^*_\delta + P \cap \mathbbmss{k}[\mathbf{t}_\delta] + \mathfrak{m}^{e_P}_\delta)_P \subseteq Q^{(*)}_P,$$ that is to say, $(\mathcal{C}^*_\delta + P \cap \mathbbmss{k}[\mathbf{t}_\delta] + \mathfrak{m}^{e_P}_\delta)_P = Q_P^{(*)}.$

Finally, since $Q_P = Q^{(*)}_{P, \delta}$ for some $\delta \in \Delta(I)$ and, by Corollary \ref{Cor BCD1}, $$Q^{(*)}_{P, \delta} = (\mathcal{C}^*_\delta + P \cap \mathbbmss{k}[\mathbf{t}_\delta] + \mathfrak{m}_\delta^{\mathrm{nil}(Q^{(*)}_{P, \delta})})_P,$$ we conclude that, $(\mathcal{C}^*_\delta + P \cap \mathbbmss{k}[\mathbf{t}_\delta] + \mathfrak{m}_\delta^{\mathrm{nil}(Q^{(*)}_{P, \delta})})_P = Q_P^{(*)}$ by the minimality of the integer $e_P,$ and we are done.\hfill$\qed$

\medskip
The proof of the following corollary follows immediately from Theorem \ref{Th BCPD2} and Lemma \ref{Lemma BPD}.

\begin{corollary}
Let $I \subset \mathbbmss{k}[\mathbf{t}]$ be a binomial ideal. For each $\delta \in \Delta(I),$ there exists  binomial canonical component $Q^{(*)}_\delta$ of $I$ with $$\mathrm{nil}(Q^{(*)}_\delta) = \mathrm{nil}(\mathcal{C}^*_\delta),$$ where $\mathcal{C}^*_\delta$ is the $\delta-$cellular canonical component of $I.$.
\end{corollary}

Therefore, we conclude that the indices of nilpotency of the canonical cellular components of binomial ideal can be interpreted as optimal bounds for the indices of nilpotency of its binomial canonical components.

\section{Examples}\label{Sect Examples}

\begin{example}
Let $\mathbbmss{k}$ be an algebraically closed field and let $I \subset \mathbbmss{k}[t_1, \ldots, t_n]$ be a binomial ideal. If $\sqrt{I}$ is prime and does not contain any of the variables, then, by \cite[Theorem 8.3]{Eisenbud96}, the canonical cellular decomposition of $I$ is the canonical (primary) decomposition of $I.$
\end{example}

\begin{example}
Given a sublattice $\mathcal{L}$ of $\mathbb{Z}^n$ and a group homomorphism $\rho : \mathcal{L} \to \mathbb{C}^*,$ we define the ideal $$I_+(\rho) = \big\{ \rho(\mathbf{u}) \mathbf{t}^\mathbf{u} - \rho(\mathbf{v}) \mathbf{t}^\mathbf{v} \mid \mathbf{u} - \mathbf{v} \in \mathcal{L} \big\}$$ in $\mathbbmss{C}[t_1, \ldots, t_n].$

If $M$ is a monomial ideal in $\mathbbmss{C}[t_1, \ldots, t_n],$ then, by (2.7) in \cite{Altmann}, the canonical cellular decomposition of $I = I_+(\rho) + M$ is the canonical (primary) decomposition of $I.$
\end{example}

\begin{example}
In this example, we study a family of ideals from \cite{Diaconis}, where it is proved that primary decompositions of these ideals provide useful descriptions of components of certain graphs arising in problems from combinatorics, statistics, and operations research.

Let $I_\mathcal{L}$ be the prime ideal generated by all $2 \times 2$-minors of
$$ \left(
\begin{array}{cccc}
 t_{11} & t_{12} & \ldots & t_{1b} \\
 t_{21} & t_{22} & \ldots & t_{2b} \\
 \vdots & \vdots & \ddots & \vdots \\
 t_{a1} & t_{a2} & \ldots & t_{ab} \\
\end{array}
\right)$$
in $\mathbbmss{k}[\{t_{ij}\}],$ where $a,b \geq 3.$ Let $R = (t_{11}, \ldots, t_{1b})$ and $C = (t_{11}, \ldots, t_{a1}).$ In \cite{Diaconis}, it is shown that the ideal of corner minors $$I_{\mathcal{B}_{\rm cor}} = \Big\langle \big\{ t_{11}t_{ij} - t_{1j}t_{i1} \mid 2 \leq i \leq a,\ 2 \leq j \leq b \big\} \Big\rangle$$ has the following minimal primary decomposition $$I_{\mathcal{B}_{\rm cor}} = I_\mathcal{L} \cap R \cap C \cap Q,$$ where $Q = I_{\mathcal{B}_{\rm cor}} + R^2 + C^2.$

Observe that the ideals $I_\mathcal{L}, R$ and $C$ are prime, so they are the corresponding canonical components of $I_{\mathcal{B}_{\rm cor}}.$

Let us prove that
\begin{equation}\label{equ4}
I_{\mathcal{B}_{\rm cor}} = I_\mathcal{L} \cap R \cap C \cap \Big( \big( I_{\mathcal{B}_{\rm cor}} + ( R + C )^3 \big) : \big( \prod_{i,j \neq 1} t_{ij} \big)^\infty \Big)
\end{equation}
is the canonical decomposition of $I_{\mathcal{B}_{\rm cor}}.$

First of all, we notice that the radical of $Q$ is $R + C.$ Moreover, $(R + C)^3 \subseteq R^2 + C^2 \subseteq Q,$ so we have that ${\rm nil}(Q) \leq 3$ and since $t_{12}t_{21} \in (R+C)^2$ does not lie in $Q,$ we conclude that ${\rm nil}(Q)=3.$

We next prove that
\begin{equation}\label{equ5}
I_\mathcal{L} \cap R \cap C \subseteq I_{\mathcal{B}_{\rm cor}} + ( R + C )^2.
\end{equation}

Let $f \in I_\mathcal{L} \cap R \cap C.$ Since $I_\mathcal{L}$ is a binomial ideal not containing any monomial, by Corollary 1.5 in \cite{Eisenbud96}, we may assume that $f$ is homogeneous of degree at least $2,$ that is, $f = m_1 - m_2$ with $\deg(m_1)=\deg(m_2) \geq 2.$ On the other hand, since $C$ is a monomial ideal and $f \in
C,$ the terms $m_1, m_2$ lie in $C.$ So we can write $m_1 = t_{i_1 1}m_{11}$ and $m_2 = t_{i_2 1}m_{12},$ with $\deg(m_{11}),\deg(m_{12}) \geq 1.$ Arguing similarly for $f \in R,$ we obtain that $m_1 = t_{1 j_1} m_{21}$ and $m_2 = t_{1 j_2}m_{22},$ with $\deg(m_{21}),\deg(m_{22}) \geq 1.$ Therefore, either $m_1 = t_{11}m_{11} = t_{11}m_{21}$ or $m_1 = t_{i_1 1}t_{1 j_1}m_{31}, $ with $i_1$ and $j_1$ not simultaneously equal to $1.$ If $m_1 = t_{11}m_{11},$ then $$t_{11}m_{11} = t_{11}t_{kl}m_{31} = (t_{11}t_{kl} - t_{k1}t_{1l})m_{31} + t_{k1}t_{1l}m_{31} \in I_{\mathcal{B}_{\rm cor}} + ( R + C )^2,$$ otherwise $m_1 \in ( R + C )^2.$ In both cases, $m_1 \in I_{\mathcal{B}_{\rm cor}} + ( R + C )^2.$ Analogously, we can prove that $m_2 \in I_{\mathcal{B}_{\rm cor}} + ( R + C )^2.$ Therefore, we conclude that $f = m_1 - m_2 \in I_{\mathcal{B}_{\rm cor}} + ( R + C )^2$ as desired.

Now, by (\ref{equ5}), we have that $I_{\mathcal{B}_{\rm cor}}$ is strictly contained in $$I_\mathcal{L} \cap R \cap C \cap \left( I_{\mathcal{B}_{\rm cor}} + ( R + C )^2 \right) ) = I_\mathcal{L} \cap R \cap C.$$ Therefore, from Corollary \ref{Cor Pre0} and Corollary \ref{Cor Pre1}, it follows that the $(R+C)-$ca\-no\-nical component of $ I_{\mathcal{B}_{\rm cor}}$ is the $(R+C)-$primary component of $\big( I_{\mathcal{B}_{\rm cor}} + ( R + C )^3 \big)$ which is nothing but $$Q^* = \left( \left( I_{\mathcal{B}_{\rm cor}} + ( R + C )^3 \right) : \left( \prod_{i,j \neq 1} t_{ij} \right)^\infty \right).$$ Observe that, if $\delta = \{t_{ij} \mid i=2, \ldots, a,\ j = 1,\ldots, b\},$ $Q^*$ is also the ca\-no\-nical $\delta-$cell\-ular of $I_{\mathcal{B}_{\rm cor}}$ (see Corollary \ref{Cor BCCD1}). In fact, we have shown that the canonical cellular decomposition of the ideal of corner minors agrees with its canonical (primary) decomposition.
\end{example}

\end{document}